\documentclass[a4paper,12pt]{article}

\usepackage{typearea}
\typearea{15}

\usepackage{mymacros}

\usepackage{comment}

\newcommand{\bp}{\bsp}
\newcommand{\lcm}{\operatorname{lcm}}
\newcommand{\myceil}[1]{\left\lceil {#1} \right\rceil}

\title{On a certain generalization of triangle singularities}
\author{Kenji Hashimoto, Hwayoung Lee, Kazushi Ueda}
\date{}
\begin{document}

\maketitle

\begin{abstract}
Triangle singularities are Fuchsian singularities
associated with von Dyck groups,
which are index two subgroups of Schwarz triangle groups.
Hypersurface triangle singularities are classified by Dolgachev,
and give $14$ exceptional unimodal singularities
classified by Arnold.
We introduce a generalization of triangle singularities
to higher dimensions,
show that there are only finitely many hypersurface singularities of this type
in each dimension, and
give a complete list in dimension $3$.
\end{abstract}

\section{Introduction}
 \label{sc:introduction}

Let $\bp = (p,q,r)$ be a triple of positive integers satisfying
\begin{align}
 \frac{1}{p} + \frac{1}{q} + \frac{1}{r} < 1.
\end{align}
The \emph{Schwarz triangle group}
\begin{align}
 \Delta = \la a, b, c \relmid a^2 = b^2 = c^2 = (ab)^p = (bc)^q = (ca)^r = e \ra
\end{align}
is the reflection group generated by reflections
along edges of the hyperbolic triangle
in the upper half plane $\bH$
with angles $\pi/p$, $\pi/q$, and $\pi/r$.
The \emph{von Dyck group}
\begin{align}
 \Gamma = \la x, y, z \relmid x^p = y^q = z^r = x y z = e \ra
\end{align}
is the subgroup of $\Delta$ of index two
consisting of products of even numbers of reflections.
It is a cocompact Fuchsian group,
and the orbifold quotient
$\bX = [\bH/\Gamma]$
has three orbifold points
with stabilizers $\bZ/p \bZ$,
$\bZ/q \bZ$, and $\bZ/r \bZ$.
The triple $\bp = (p,q,r)$ is called the \emph{signature}
of the Fuchsian group $\Gamma$.

Smooth rational orbifolds are studied in detail
by Geigle and Lenzing \cite{Geigle-Lenzing_WPC}
under the name of \emph{weighted projective lines}.
In particular,
the orbifold $\bX$ can be described as
\begin{align}
 \bX = [ (\Spec T \setminus \bszero) / K],
\end{align}
where
\begin{align}
 T
  = \bC[X,Y,Z]/(X^p+Y^q+Z^r)
  = \bigoplus_{\veck \in L} T_\veck
\end{align}
is a two-dimensional ring
graded by the abelian group
\begin{align}
 L = \bZ \vecX \oplus \bZ \vecY \oplus \bZ \vecZ \oplus \bZ \vecc /
  (p \vecX - \vecc, q \vecY - \vecc, r \vecZ - \vecc)
\end{align}
of rank one,
which is the group of characters of the algebraic group
\begin{align}
 K = \Spec \bC[L].
\end{align}
Here the grading is given by $X \in T_\vecX$, etc.
The \emph{dualizing element} is defined by
\begin{align}
 \vecomega = \vecc - \vecX - \vecY - \vecZ,
\end{align}
and the canonical ring of the orbifold $\bX$ is given by
\begin{align}
 R = \bigoplus_{k=0}^\infty
  R_k, \qquad R_k = 
  T_{k \vecomega}.
\end{align}
The isolated singularity at the origin of the scheme $\Spec R$
is called the \emph{triangle singularity}
with signature $\bp=(p, q, r)$.

Let
$
 \frakm
  = \bigoplus_{k=1}^\infty R_k 
$
denote the irrelevant ideal of $R$.
The dimension of the vector space $\frakm / \frakm^2$ is called
the \emph{embedding dimension}
of $R$.
It is known that the embedding dimension of $R$ coincides with the minimum number of generators of $R$ (see \pref{lm:embdim}).
A graded ring is said to be a \emph{hypersurface}
if the embedding dimension is greater
than the Krull dimension by one.

\begin{theorem}[{\cite{MR0568895}, cf.~also~\cite{Milnor_3BM,MR586722}}]
 \label{th:Dolgachev}
The ring $R$ is a hypersurface
if and only if the signature $\bp$ is one of the 14 signatures
shown in \pref{tb:Dolgachev}.
\end{theorem}

\begin{table}[t]
\begin{align*}
\begin{array}{cccc}
\toprule
\text{Signature} & \text{Generators} & \text{Weights} & \text{Relation} \\
\midrule
 (2, 3, 7) & (X,Y,Z) & (21, 14, 6;42) & x^2+y^3+z^7 \\
 (2, 3, 8) & (XZ,Y,Z^2) & (15, 8, 6;30) & x^2+y^3 z+z^5 \\
 (2, 3, 9) & (X,YZ,Z^3) & (9, 8, 6;24) & x^2 z+y^3+z^4 \\
 (2, 4, 5) & (XY,Y^2,Z) & (15, 10, 4;30) & x^2+y^3+y z^5 \\
 (2, 4, 6) & (XYZ,Y^2,Z^2) & (11, 6, 4;22) & x^2+y^3 z+y z^4 \\
 (2, 4, 7) & (XY,Y^2 Z,Z^3) & (7, 6, 4;18) & x^2 z+y^3+y z^3 \\
 (2, 5, 5) & (Y^5,X,YZ) & (10, 5, 4;20) & x^2+x y^2+z^5 \\
 (2, 5, 6) & (Y^4,XZ,Y Z^2) & (6, 5, 4;16) & x y^2+x^2 z+z^4 \\
 (3, 3, 4) & (X^3,XY,Z) & (12, 8, 3;24) & x^2+x z^4+y^3 \\
 (3, 3, 5) & (X^3 Z,XY,Z^2) & (9, 5, 3;18) & x^2+x z^3+y^3 z \\
 (3, 3, 6) & (X^3,XYZ,Z^3) & (6, 5, 3;15) & x^2 z+x z^3+y^3 \\
 (3, 4, 4) & (X Y^4,X^2,YZ) & (8, 4, 3;16) & x^2+x y^2+y z^4 \\
 (3, 4, 5) & (X Y^3,X^2 Z,Y Z^2) & (5, 4, 3;13) & x y^2+x^2 z+y z^3 \\
 (4, 4, 4) & (X^4,Y^4,XYZ) & (4, 4, 3;12) & x^2 y+x y^2+z^4 \\
\bottomrule
\end{array}
\end{align*}
\caption{Hypersurface triangle singularities}
\label{tb:Dolgachev}
\end{table}

\pref{tb:Dolgachev} coincides
with the list of weighted homogeneous exceptional unimodal singularities
classified by Arnold \cite{Arnold_ICM}.
The signature in this context is called the \emph{Dolgachev number}
of the exceptional unimodal singularity.
Besides the Dolgachev number,
each exceptional unimodal singularity comes with another triple of positive integers
called the \emph{Gabrielov number},
which describes the Milnor lattice of the singularity
\cite{MR0367274}.
This leads to the discovery of \emph{strange duality}
by Arnold,
which states that exceptional unimodal singularities come in pairs
in such a way that the Dolgachev number and the Gabrielov number are interchanged.
Strange duality is described in terms of an exchange of the algebraic lattice
and the transcendental lattice of a K3 surface
by Dolgachev and Nikulin \cite{Dolgachev_IQF,Nikulin_ISBF}
and Pinkham \cite{Pinkham_strange-duality},
and is now considered as a precursor of \emph{mirror symmetry}.

In this paper,
we consider the following generalization
of triangle singularities.
Let $n$ be an integer greater than 3 and
$\bp = (p_1, \ldots, p_n)$ be a sequence of integers
called a \emph{signature}.
In what follows, we assume that $\bp$ satisfies
\begin{align}
 \sum_{i=1}^n \frac{1}{p_i} < 1.
\end{align}
Consider the ring
\begin{align}
 T = \bC[X_1, \ldots, X_n] \left/ \lb \sum_{i=1}^n X_i^{p_i} \rb \right.,
\end{align}
which is graded by the abelian group
\begin{align} \label{eq:L}
 L = \bZ \vecX_1 \oplus \cdots \oplus \bZ \vecX_n \oplus \bZ \vecc \left/ \lb p_i \vecX_i - \vecc \rb \right.
\end{align}
of rank one.
Here the grading is given by $X_i \in T_{\vecX_i}$ for all $i \in \{ 1, \ldots, n \}$.
The \emph{dualizing element} is defined by
\begin{align}
 \vecomega = \vecc - \sum_{i=1}^n \vecX_i,
\end{align}
and the \emph{canonical ring} is defined by
\begin{align}
 R =  \bigoplus_{k=0}^\infty
 R_k, \qquad R_k = 
 T_{k \vecomega}.
\end{align}
The singularity of $\Spec R$ is
a higher-dimensional generalization of triangle singularities.
The main result in this paper is the finiteness
of hypersurface singularities of this type:

\begin{theorem} \label{th:main}
For any integer $n$ greater than 3,
there are only finitely many signatures $\bp = (p_1,\ldots,p_n)$
such that $R$ is a hypersurface.
\end{theorem}

Our proof of \pref{th:main} gives an algorithm
to classify all hypersurface generalized triangle singularities
for any given $n \ge 4$.
The list of hypersurface generalized triangle singularities for $n=4$
is shown in \pref{tb:dim3}.

Our proof of \pref{th:main} also gives the following:

\begin{theorem} \label{th:isolated}
For any integer $n$ greater than 3,
the generalized triangle singularity
associated with signature $\bp = (p_1, \ldots, p_n)$
is an isolated hypersurface singularity
if and only if
\begin{align} \label{eq:fraction}
 \sum_{i=1}^n \frac{1}{p_i} + \frac{1}{\prod_{i=1}^n p_i} = 1.
\end{align}
\end{theorem}

A slight generalization
\begin{align} \label{eq:fraction2}
 \sum_{i=1}^n \frac{1}{p_i} + \frac{1}{\prod_{i=1}^n p_i} \in \bZ
\end{align}
of the Diophantine equation
\eqref{eq:fraction} is equivalent to the \emph{improper Zn\'am problem},
which asks for a sequence $(p_1, \ldots, p_n)$
of integers satisfying
$p_i \mid \prod_{j \in \{ 1, \ldots, n \} \setminus \{ i \}} p_j + 1$
for all $i \in \{ 1, \ldots, n \}$.
It is not known whether \eqref{eq:fraction2} implies \eqref{eq:fraction}.

The \emph{$a$-invariant} $a=a(R) \in \bZ$
of a graded Gorenstein ring $R$
is defined by
\begin{align}
 K_R = R(a),
\end{align}
where $K_R$ is the graded canonical module of $R$.
If $R$ is a hypersurface
generated by $n$ elements of degrees $a_1, \ldots, a_n$
with one relation of degree $h$,
then the $a$-invariant of $R$ is given by
\begin{align}
 a = h-a_1-\cdots-a_n.
\end{align}

\begin{theorem} \label{th:a}
If $R$ is a hypersurface generalized triangle singularity,
then one has
\begin{align} \label{eq:a}
 a(R) = 1.
\end{align}
\end{theorem}

The equality \eqref{eq:a} also holds for
hypersurface (ordinary) triangle singularities (cf.\ \cite[Proposition 2.8]{MR586722}).
It follows from \eqref{eq:a} that one can compactify $\Spec R$
to a weighted projective hypersurface
with trivial canonical bundle
by adding one variable of degree one.
It is an interesting problem to study
when this hypersurface admits a smoothing.

\ \\
\emph{Acknowledgements}:
We thank Korea Institute for Advanced Study
for financial support and excellent research environment.
K.~U. is supported by JSPS Grant-in-Aid for Young Scientists No.~24740043.

\begin{table}[htbp]
\begin{align*}
\begin{array}{cccc}
\toprule
\text{Signature} & \text{Generators} & \text{Weights} & \text{Relation} \\
\midrule
( 2 , 3 , 7 , 43 )  & (X,Y,Z,W)       & ( 903 , 602 , 258 , 42 ; 1806 ) & x^2+y^3+z^7+w^{43} \\
( 2 , 3 , 7 , 44 )  & (XW,Y,Z,W^2)    & ( 483 , 308 , 132 , 42 ; 966 )  & x^2+w(y^3+z^7+w^{22}) \\
( 2 , 3 , 7 , 45 )  & (X,YW,Z,W^3)    & ( 315 , 224 , 90 , 42 ; 672 )   & y^3+w(x^2+z^7+w^{15}) \\
( 2 , 3 , 7 , 49 )  & (X,Y,ZW,W^7)    & ( 147 , 98 , 48 , 42 ; 336 )    & z^7+w(x^2+y^3+w^7) \\
( 2 , 3 , 8 , 25 )  & (XZ,Y,Z^2,W)    & ( 375 , 200 , 150 , 24 ; 750 ) & x^2+z(y^3+z^4+w^{25}) \\
( 2 , 3 , 8 , 26 )  & (XZW,Y,Z^2,W^2) & ( 207 , 104 , 78 , 24 ; 414 ) & x^2+zw(y^3+z^4+w^{13}) \\
( 2 , 3 , 9 , 19 )  & (X,YZ,Z^3,W)    & ( 171 , 152 , 114 , 18 ; 456 ) & y^3+z(x^2+z^3+w^{19}) \\
( 2 , 3 , 9 , 21 )  & (X,YZW,Z^3,W^3) & ( 63 , 62 , 42 , 18 ; 186 ) & y^3+zw(x^2+z^3+w^{7}) \\
( 2 , 3 , 10 , 16 ) & (XZW,Y,Z^2,W^2) & ( 159 , 80 , 48 , 30 ; 318 ) & x^2+zw(y^3+z^5+w^8) \\
( 2 , 3 , 13 , 13 ) & (W^{13},X,Y,ZW) & ( 78 , 39 , 26 , 12 ; 156 ) & w^{13}+x(x+y^2+z^3) \\
( 2 , 4 , 5 , 21 )  & (XY,Y^2,Z,W)       & ( 315 , 210 , 84 , 20 ; 630 ) & x^{2}+y(y^2+z^5+w^{21}) \\
( 2 , 4 , 5 , 22 )  & (XYW,Y^2,Z,W^2)    & ( 175 , 110 , 44 , 20 ; 350 ) & x^{2}+yw(y^2+z^5+w^{11}) \\
( 2 , 4 , 6 , 13 )  & (XYZ,Y^2,Z^2,W)    & ( 143 , 78 , 52 , 12 ; 286 ) & x^2+yz(y^2+z^3+w^{13}) \\
( 2 , 4 , 6 , 14 )  & (XYZW,Y^2,Z^2,W^2) & ( 83 , 42 , 28 , 12 ; 166 ) & x^2+yzw(y^2+z^3+w^{7}) \\
( 2 , 4 , 7 , 10 )  & (XYW,Y^2,W^2,Z)    & ( 119 , 70 , 28 , 20 ; 238 ) & x^2+yz(y^2+z^5+w^{7}) \\
( 2 , 5 , 5 , 11 )  & (Y^5,X,YZ,W)       & ( 110 , 55 , 44 , 10 ; 220 ) & z^5+x(x+y^2+w^{11}) \\
( 2 , 5 , 5 , 15 )  & (Y^5,X,YZW,W^5)    & ( 30 , 15 , 14 , 10 ; 70 ) & z^5+xw(x+y^2+w^{3}) \\
( 2 , 5 , 6 , 8 )   & (XZW,Z^2,W^2,Y)    & ( 95 , 40 , 30 , 24 ; 190 ) & x^2+yz(y^3+z^4+w^{5}) \\
( 2 , 5 , 7 , 7 )   & (W^7,X,ZW,Y)       & ( 70 , 35 , 20 , 14 ; 140 ) & z^7+x(x+y^2+w^{5}) \\
( 2 , 7 , 7 , 7 )   & (Z^7,Y^7,X,YZW)    & ( 14 , 14 , 7 , 6 ; 42 ) & w^7+xy(x+y+z^2) \\
( 3 , 3 , 4 , 13 )  & (X^3,XY,Z,W)       & ( 156 , 104 , 39 , 12 ; 312 ) & y^3+x(x+z^4+w^{13}) \\
( 3 , 3 , 4 , 15 )  & (X^3,XYW,Z,W^3)    & ( 60 , 44 , 15 , 12 ; 132 ) & y^3+xw(x+z^4+w^{5}) \\
( 3 , 3 , 5 , 8 )   & (X^3,XY,Z,W)       & ( 120 , 80 , 24 , 15 ; 240 ) & y^3+x(x+z^5+w^{8}) \\
( 3 , 3 , 5 , 9 )   & (X^3,XYW,W^3,Z)    & ( 45 , 35 , 15 , 9 ; 105 ) & y^3+xz(x+z^3+w^{5}) \\
( 3 , 3 , 6 , 7 )   & (X^3,XYZ,Z^3,W)    & ( 42 , 35 , 21 , 6 ; 105 ) & y^3+xz(x+z^2+w^{7}) \\
( 3 , 3 , 6 , 9 )   & (X^3,XYZW,Z^3,W^3) & ( 18 , 17 , 9 , 6 ; 51 ) & y^3+xzw(x+z^2+w^{3}) \\
( 3 , 4 , 4 , 8 )   & (Z^4,YZW,W^4,X)    & ( 24 , 15 , 12 , 8 ; 60 ) & y^4+xz(x+z^2+w^{3}) \\
( 3 , 4 , 5 , 5 )   & (W^5,ZW,X,Y)       & ( 60 , 24 , 20 , 15 ; 120 ) & y^5+x(x+z^3+w^{4}) \\
( 3 , 5 , 5 , 5 )   & (W^5,Z^5,YZW,X)    & ( 15 , 15 , 9 , 5 ; 45 ) & z^5+xy(x+y+w^3) \\
( 4 , 4 , 4 , 5 )   & (Y^4,Z^4,XYZ,W)     & ( 20 , 20 , 15 , 4 ; 60 ) & z^4+xy(x+y+w^5) \\
( 4 , 4 , 4 , 8 )   & (X^4,Y^4,XYZW,W^4) & ( 8 , 8 , 7 , 4 ; 28 )  & z^4+xyw(x+y+w^2) \\
( 5 , 5 , 5 , 5 )   & (X^5,Y^5,Z^5,XYZW) & ( 5 , 5 , 5 , 4 ; 20 ) & w^5+xyz(x+y+z) \\
\bottomrule
\end{array}
\end{align*}
\caption{Hypersurface generalized triangle singularities in dimension 3}
\label{tb:dim3}
\end{table}


\clearpage

\section{Proof of \pref{th:main}}
 \label{sc:main}

Let $\bp = (p_1, \ldots, p_n)$ be a signature
such that $R$ is a hypersurface, and
\begin{align}
 \Tbar
  = \bigoplus_{\vecv \in L} \Tbar_\vecv
  = \bC \ld \Xbar_1,\ldots,\Xbar_n \rd
\end{align}
be a polynomial ring
graded by the abelian group $L$
in \eqref{eq:L}.
The Veronese subring of $\Tbar$
over $\bZ \, \vecomega$ is denoted by
\begin{align}
 \Rbar = \bigoplus_{k \in \bZ} \Rbar_k, \qquad
 \Rbar_k = \Tbar_{k \vecomega}.
\end{align}
We write
\begin{align}
 [i,j]=\{i,i+1,\ldots,j\}
\end{align}
for a pair $(i,j)$ of integers with $i \le j$.
Let
$
 \varphi \colon \Tbar \to T = \Tbar/ \lb \Fbar \rb
$
be the natural projection,
where
\begin{align}
 \Fbar
  = \sum_{i=1}^n \Ybar_i,
  \qquad
 \Ybar_i = \Xbar_i^{p_i}
 \text{ for } i \in [1,n].
\end{align}
The restriction
$
 \varphi|_\Rbar \colon \Rbar \to R = \Rbar / \lb \Fbar \rb
$
will also be denoted by $\varphi$
by abuse of notation.
We write $X_i=\varphi(\Xbar_i)$ for $i \in [1,n]$.
Any element $\vecv \in L$ can be written uniquely as
\begin{align} \label{eq:ka}
 \vecv = \ell \vecc + \sum_{i=1}^n a_i \vecX_i,
\end{align}
where $\ell \in \bZ$ and
$0 \le a_i \le p_i-1$ for any $i \in [1,n]$.
This defines functions $\ell \colon L \to \bZ$
and $a_i \colon L \to [0,p_i-1]$
for $i \in [1,n]$.
Any element of $\Tbar_\vecv$
for $\vecv \in L$ can be written as the product
$
 \Mbar(\vecv) \Pbar
$, where $\Mbar(\vecv)$ is the monomial defined by
\begin{align} \label{eq:Mbar}
 \Mbar(\vecv) = \prod_{i=1}^n \Xbar_i^{a_i(\vecv)}
\end{align}
and
$
 \Pbar \in \bC \ld \Ybar_1,\ldots,\Ybar_n \rd_\ell
$
is a homogeneous polynomial of degree $\ell$.
\begin{comment}
\begin{shaded}
Since $\{ Y_i \}_{i=1}^n$ has one relation
\begin{align}
 \sum_{i=1}^n Y_i = 0,
\end{align}
we have
\begin{align}
 \dim R_k = \binom{\ell(k \vecomega)+n-1}{n-1}
\end{align}
where
\begin{align}
 \ell(k \vecomega)
  = k - \sum_{i=1}^n \myceil{\frac{k}{p_i}}.
\end{align}
We also set
$
 M = \varphi \circ \Mbar \colon L \to T.
$
\end{shaded}
\end{comment}

Let
$
 \frakm
  = \bigoplus_{k=1}^\infty R_{k}
$
be the irrelevant ideal of $R$.
Since we assume that $R$ is a hypersurface,
 that is, $\dim_\bC \frakm/\frakm^2 = n$,
 we can choose a set
$
 \Xibar = \lc \xbar_i \rc_{i=1}^n \subset \Rbar
$
of monomials
whose image
$
 \Xi = \varphi \lb \Xibar \rb
  = \lc x_i = \varphi \lb \xbar_i \rb \rc_{i=1}^n
$
forms a basis of the vector space
$\frakm/\frakm^2$.
By the following \pref{lm:embdim}, the ring $R$ is generated by $\Xi$ over $\bC$.

\begin{lemma} \label{lm:embdim}
Let $S$ be a subset of $\frakm$ consisting of homogeneous elements,
 that is, $S \subset \bigcup_{k=1}^\infty R_k$.
Then $S$ generates the ring $R$ over $\bC$ if and only if
 the image of $S$ spans the vector space $\frakm / \frakm^2$.
\end{lemma}
\begin{proof}
Let $\bC S$ denote the vector space spanned by $S$.
If $S$ generates $R$, that is, $\bC[S]=R$, then
 $\frakm/\frakm^2=(\bC[S]\cap \frakm)/\frakm^2=(\bC S+\frakm^2)/\frakm^2$.
Hence the image of $S$ spans $\frakm/\frakm^2$.
Conversely, assume that the image of $S$ spans $\frakm/\frakm^2$.
In order to prove that $S$ generates $R$ by induction,
 suppose that $\bC[S]$ contains $R_k$ for all $k<n$.
Then $\bC[S] \supset \frakm^2 \cap R_n$.
Since the image of $S$ spans $\frakm/\frakm^2$, we have
 $R_n=\bC S \cap R_n + \frakm^2 \cap R_n$,
 which implies that $\bC[S]$ contains $R_n$.
Therefore we obtain $\bC[S]=R$ by induction.
\end{proof}

We also set
\begin{align}  \label{eq:df}
\begin{split}
 \nu=1-\sum_{i=1}^n \frac{1}{p_i}, \quad
 N= \lcm \lc p_i \relmid i \in [1,n] \rc, \\
 N_i=\lcm \lc p_j \relmid j \in [1,n] \setminus \{ i \} \rc, \quad
 q_i=\nu p_i N_i.
\end{split}
\end{align}

\begin{lemma} \label{lm:power}
For any $i \in [1,n]$ and any $m \in \bN$,
we have $\Xbar_i^m \in \Rbar$ if and only if $q_i \mid m$.
\end{lemma}

\begin{proof}
Note that $\Rbar_k$ contains a pure power of $\Xbar_i$
if and only if
\begin{align} \label{eq:a=0}
 a_j(k \vecomega) = 0
 \qquad \text{for any } j \in [1,n] \setminus \{i\},
\end{align}
where $a_j \colon L \to [0,p_j-1]$ are defined
by \eqref{eq:ka}.
Since 
\begin{align}
 a_j(k \vecomega) = p_j \myceil{\frac{k}{p_j}} - k
  \qquad \text{for any } j \in [1,n],
\end{align}
the condition \eqref{eq:a=0} holds
if and only if $k$ is an integer multiple of $N_i$.
It follows from
\begin{align}
 N_i \vecomega
 =N_i \vecc -
  \sum_{j \in [1,n] \setminus \{ i \}} \frac{N_i}{p_j} \vecc
  - N_i \vecX_i
 &= \lb \nu + \frac{1}{p_i} \rb N_i \vecc - N_i \vecX_i
 = q_i \vecX_i
\end{align}
that the pure power of $\Xbar_i$ contained in $\Rbar_{N_i}$ is $\Xbar_i^{q_i}$, and
\pref{lm:power} is proved.
\end{proof}

\begin{corollary} \label{cr:puregen}
If $\Xbar_i^m \in \Xibar$ for some $i \in [1,n]$ and $m \in \bN$,
then one has $m=q_i$.
\end{corollary}

\begin{proof}
This is immediate from \pref{lm:power}.
\end{proof}

\begin{lemma} \label{lm:MPQ}
For any $k \in \bN$
and any $\Zbar \in \Rbar_k$,
there exist
$
 \Pbar
  \in \bC \ld \Ybar_1,\ldots,\Ybar_n \rd_{\ell(k \vecomega)-1}
$
and
$
 \Qbar \in \bC \ld \xbar_1,\ldots,\xbar_n \rd \cap \Rbar_k
$
such that
$
 \Zbar = \Mbar(k \vecomega) \Pbar \Fbar + \Qbar.
$
\end{lemma}

\begin{proof}
Since $\Xi$ generates $R$ as a ring,
there exist $\Pbar' \in \Tbar_{k \vecomega - \vecc}$
and $\Qbar \in \bC \ld \xbar_1, \ldots, \xbar_n \rd$
such that
$
 \Zbar = \Pbar' \Fbar + \Qbar.
$
It follows from the definition of $\ell$ and $\Mbar$ that
$
 \ell(k \vecomega-\vecc) = \ell(k \vecomega)-1
$
and
$
 \Mbar(k \vecomega-\vecc) = \Mbar(k \vecomega).
$
Hence there exists
$
 \Pbar \in \bC \ld \Ybar_1,\ldots,\Ybar_n \rd_{\ell(k \vecomega)-1}
$
such that
$
 \Pbar' = \Mbar(k \vecomega) \Pbar.
$
\end{proof}

\begin{lemma} \label{lm:R=T}
If $\xbar_i=\Xbar_i^{q_i}$ for all $i \in [1,n]$,
then we have $R=T$ and $\xbar_i=\Xbar_i$ for all $i \in [1,n]$.
\end{lemma}

\begin{proof}
For each $i$, fix a sufficiently large $k$
with $k\equiv -1 \bmod p_i$.
Since
\begin{align}
 k \vecomega = k \vecc - k \sum_{i=1}^n \vecX_i,
\end{align}
we have
$
 a_i(k \vecomega) = 1,
$
so that there exists a monomial $\Gbar \in \Tbar_{k \vecomega - \vecX_i}$
such that $\Xbar_i \Gbar \in \Rbar_k$ and $\Xbar_i \nmid \Gbar$.
By applying \pref{lm:MPQ}, we have
\begin{align} \label{eq:XG2}
 \Xbar_i \Gbar = \Mbar(k\vecomega) \Pbar \Fbar + \Qbar, \quad
 \Xbar_i \mid \Mbar(k\vecomega), \quad \text{and} \quad \Xbar_i^2 \nmid \Mbar(k\vecomega).
\end{align}
Assume for contradiction that $q_i>1$.
By comparing terms of degree $1$ in the variable $X_i$
in \eqref{eq:XG2}, we obtain
\begin{align} \label{eq:XG}
 \Xbar_i \Gbar
  = \Mbar(k\vecomega) \cdot \Pbar \big|_{\Xbar_i=0} \cdot \Fbar \big|_{\Xbar_i=0}.
\end{align}
Since
$
 \Fbar \big|_{\Xbar_i=0}
  = \sum_{j \in [1,n] \setminus \{ i \}} \Xbar_j^{p_j}
$
is not a monomial,
the right hand side of \eqref{eq:XG} is not a monomial.
This contradicts the fact that the left hand side is a monomial,
and \pref{lm:R=T} is proved.
\end{proof}

\begin{lemma} \label{lm:XiXj}
If there exist $i, j \in [1,n]$ such that
$i \ne j$, $\Xbar_i^{q_i} \nin \Xibar$, and
$\Xbar_i^a \Xbar_j^b \nin \Xibar$ for all $a,b \ge 1$,
then $\Xbar_j^{q_j} \in \Xibar$ and $p_i \mid q_i$.
\end{lemma}

\begin{proof}
By \pref{lm:power},
we have $\Xbar_i^{q_i} \in R$.
Hence we have
\begin{align} \label{eq:power_x_k}
 \Xbar_i^{q_i} = \Mbar \lb q_i \vecX_i \rb \Pbar \Fbar + \Qbar
\end{align}
by \pref{lm:MPQ}.
Let
$
 \pi \colon \bC \ld \Xbar_1, \ldots, \Xbar_n \rd
  \to \bC \ld \Xbar_i, \Xbar_j \rd
$
be the surjective ring homomorphism
defined by
\begin{align}
 \pi \lb \Xbar_k \rb =
\begin{cases}
 \Xbar_k & k = i, j, \\
 0 & \text{otherwise}.
\end{cases}
\end{align}
By projecting \eqref{eq:power_x_k} by $\pi$,
we obtain
\begin{align} \label{eq:Xq}
 \Xbar_i^{q_i}
 = \pi \lb M \lb q_i \vecX_i \rb \Pbar \rb \cdot \lb \Xbar_i^{p_i}
  + \Xbar_j^{p_j} \rb + \pi \lb \Qbar \rb.
\end{align}
It follows from the assumption of \pref{lm:XiXj}
that the only element in $\Xibar$
whose image by $\pi$ does not vanish
is a polynomial in $\Xbar_j$.
Hence we have $\pi \lb \Qbar \rb \in \bC \ld \Xbar_j \rd$.
If $\pi \lb \Qbar \rb = 0$,
then the right hand side of \eqref{eq:Xq} is not a monomial,
which contradicts the fact that the left hand side is a monomial.
Hence we have $\pi \lb \Qbar \rb \ne 0$,
so that $\Xbar_j^m \in \Xibar$ for some $m \in \bN$.
This implies $m=q_j$ by \pref{cr:puregen}.
It follows from \eqref{eq:Xq} that
$\Xbar_i^{q_i}-\pi \lb \Qbar \rb$ is divisible
by $\pi \lb \Mbar(q_i \vecX_i) \rb$.
Together with the fact that $\pi \lb \Qbar \rb \in \bC \ld \Xbar_j \rd$,
this implies that $\Mbar \lb q_i \vecX_i \rb = 1$.
Hence $p_i$ divides $q_i$, and
\pref{lm:XiXj} is proved.
\end{proof}

\begin{lemma} \label{lm:XiXjXk}
Let $i,j,k$ be distinct elements of $[1,n]$.
If $\Xbar_i^{q_i}, \Xbar_j^{q_j} \nin \Xibar$ and $\Xbar_k^{q_k} \in \Xibar$,
then there exists an element in $\Xibar$
of the form $\Xbar_i^a \Xbar_j^b \Xbar_k^c$
with $(a,b) \ne (0,0)$ and $c \ge 1$.
\end{lemma}

\begin{proof}
Assume for contradiction that $\Xbar_i^a \Xbar_j^b \Xbar_k^c \nin \Xibar$
for all $(a,b,c)$ with $(a,b)\ne (0,0)$ and $c \ge 1$.
In particular, we have $\Xbar_i^a \Xbar_k^c \nin \Xibar$ for all $a, c \ge 1$.
Together with the assumption that $\Xbar_i^{q_i} \nin \Xibar$,
this implies $p_i \mid q_i$ by \pref{lm:XiXj}.
If we set $q=q_i/p_i$,
then we have
\begin{align}
 \Ybar_i^{q} = \Pbar \Fbar + \Qbar
\end{align}
by \pref{lm:MPQ},
since $\Ybar_i^q \in T_{q \vecc}$ and $\Mbar(q \vecc) = 1$.
Let
$
 \pi \colon \bC \ld \Xbar_1, \ldots, \Xbar_n \rd
  \to \bC \ld \Xbar_i, \Xbar_j, \Xbar_k \rd
$
be the surjective homomorphism
defined by
\begin{align}
 \pi(\Xbar_l) = 
\begin{cases}
 \Xbar_l & l = i, j, k, \\
 0 & \text{otherwise}.
\end{cases}
\end{align}
We can write
\begin{align}
 \pi \lb \Pbar \rb
 = \Ybar_i^{q-1} + \Abar_1 \Ybar_i^{q-2} + \cdots + \Abar_{q-1},
\end{align}
where
$
 \Abar_l \in \bC \ld \Ybar_j,\Ybar_k \rd
$
for $l \in [1,q-1]$.
Then we have
\begin{align*}
 \pi \lb \Pbar \Fbar \rb
 &= \pi \lb \Pbar \rb \pi \lb \Fbar \rb \\
 &= \lb \Ybar_i^{q-1} + \Abar_1 \Ybar_i^{q-2} + \cdots + \Abar_{q-1} \rb
  \lb \Ybar_i+\Ybar_j+\Ybar_k \rb \\
 &= \Ybar_i^q + \lb \Abar_1+\Ybar_j+\Ybar_k \rb \Ybar_i^{q-1}
  +\lb \Abar_2 + \lb \Ybar_j+\Ybar_k \rb \Abar_1 \rb \Ybar_i^{q-2}
  +\cdots+ \lb \Ybar_j+\Ybar_k \rb \Abar_{q-1}.
\end{align*}
We shall show that $\Xbar_j \mid \Abar_l$ for all $l \in [1, q-1]$.
It follows from the assumption
that every monomial appearing in $\pi(\Qbar)$
is either divisible by $\Xbar_i \Xbar_j$ or
consists only of $\Xbar_k$.
Since all monomials appearing in
$
 \lb \Ybar_j+\Ybar_k \rb \Abar_{q-1}
$
are not divisible by $\Xbar_i$,
they must be in $\bC[\Xbar_k]$.
This implies that $\Abar_{q-1}=0$.
Since all monomials appearing in
$
 \lb \Abar_{q-1} + \lb \Ybar_j + \Ybar_k \rb \Abar_{q-2} \rb \Ybar_i
$
contains $\Xbar_i$,
they must be divisible by $\Xbar_i \Xbar_j$.
Hence we must have $\Xbar_j \mid \Abar_{q-2}$.
By repeating the same argument,
we obtain $\Xbar_j \mid \Abar_l$ for all $l \in [1, q-1]$.
In particular, one has $\Xbar_j \mid \Abar_1$.
It follows that the monomial $\Ybar_i^{q-1} \Ybar_k$
from $(\Abar_1+\Ybar_j+\Ybar_k) \Ybar_i^{q-1}$ do not cancel with any other terms.
Since this monomial is neither divisible by $\Xbar_i \Xbar_j$
nor consists only of $\Xbar_k$,
this is a contradiction,
and \pref{lm:XiXjXk} is proved.
\end{proof}

The following lemma is the key to proving \pref{th:main}.
Set
\begin{align} \label{eq:I1}
 I := \lc i \in [1,n] \relmid \Xbar_i^{q_i} \in \Xibar \rc.
\end{align}

\begin{lemma} \label{lm:gen}
If $n \ge 4$,
then we have $|I| \ge n-1$.
\end{lemma}

\begin{proof}
Assume for contradiction
that $n \ge 4$ and $r :=n - |I| \ge 2$.
If $i \ne j$ and $i, j \nin I$,
then we have $\Xbar_i^a \Xbar_j^b \in \Xibar$ for some $a, b \ge 1$
by \pref{lm:XiXj}.
It follows that
\begin{align}
 \# \lc \Xbar_i^a \Xbar_j^b \in \Xibar \relmid i,j \nin I, \ 
 a,b \ge 1 \rc
  \ge \binom{r}{2}.
\end{align}
Similarly, \pref{lm:XiXjXk} implies that
\begin{align} \label{eq:I3}
 \# \lc \Xbar_i^a \Xbar_j^b \Xbar_k^c \in \Xibar \relmid
  i,j \nin I, \ 
  k \in I, \
  (a,b) \ne (0,0), \ 
  c \ge 1 
 \rc
 \ge |I| = n - r.
\end{align}
It follows from \pref{eq:I1}--\pref{eq:I3} that
\begin{align}
 \# \Xibar
 = n
  \ge |I|+\binom{r}{2}+|I|
  = (n-r) + \frac{1}{2}r(r-1)+(n-r),
\end{align}
and hence
\begin{align} \label{eq:nr}
 n \le - \frac{1}{2}r(r-5).
\end{align}
Since
\begin{align}
 \max \lc -\frac{1}{2}r(r-5) \relmid r \in [0,n] \rc
  = 3,
\end{align}
the inequality \eqref{eq:nr} contradicts the assumption $n \ge 4$,
and \pref{lm:gen} is proved.
\end{proof}

\begin{proposition} \label{pr:nu_N}
If $n\geq 4$, we have $\nu=1/N$.
\end{proposition}

\begin{proof}
If $R=T$, we have $q_i=1$ for all $i \in [1,n]$ by \pref{lm:power}.
Hence
\begin{align}
 q_i
  =\nu p_i N_i
  =\nu \cdot \lcm(p_i,N_i) \cdot \gcd(p_i,N_i)
  =\nu N \cdot \gcd(p_i, N_i)
  =1.
\end{align}
Since $\nu N$ is an integer,
it follows that
\begin{align} \label{eq:nuN}
 \nu N
  = \gcd (p_i, N_i)
  =1.
\end{align}

If $R \subsetneq T$,
then we have $\Xbar_i^{q_i} \nin \Xibar$
for some $i \in [1,n]$ by \pref{lm:R=T}.
Since we have $\Xbar_l^{q_l}\in \Xibar$ for all $l \in [1,n] \setminus \{ i \}$
by \pref{lm:gen},
we can set $\xbar_l = \Xbar_l^{q_l}$ for $l \in [1,n] \setminus \{ i \}$.
Then we have $\xbar_i \ne \Xbar_i^{q_i}$.
Let $m$ be an element of $[1,n] \setminus \{ i \}$
such that the variable $\Xbar_m$ appears in $\xbar_i$,
and fix any distinct elements $j$ and $k$ of $[1,n] \setminus \{ i, m \}$.
Then it follows that $\Xbar_i^a \Xbar_j^b\nin \Xibar$ for all $a,b \ge 1$.
This implies that $p_i \mid q_i$ by \pref{lm:XiXj}.

Let
$
 \pi \colon \bC \ld \Xbar_1, \ldots, \Xbar_n \rd
  \to \bC \ld \Xbar_i, \Xbar_j, \Xbar_k \rd
$
be the surjective homomorphism
defined by
\begin{align}
 \pi \lb \Xbar_l \rb = 
\begin{cases}
 \Xbar_l & l = i, j, k, \\
 0 & \text{otherwise},
\end{cases}
\end{align}
just as in the proof of \pref{lm:XiXjXk}.
It follows from the choice of $j$ and $k$ that
\begin{align} \label{eq:pixi}
 \pi \lb \xbar_l \rb = 0 \text{ for any } l \in [1,n] \setminus \{ j, k \}.
\end{align}
If we write $q=q_i/p_i$,
then the same argument as in the proof of \pref{lm:XiXjXk}
shows
\begin{align} \label{eq:piQ}
 \Ybar_i^q = \Pbar \Fbar + \Qbar
\end{align}
and
\begin{align} \label{eq:PF}
 \pi(\Pbar \Fbar)
 &= \Ybar_i^q + \lb \Abar_1 + \Ybar_j + \Ybar_k \rb \Ybar_i^{q-1}
  + \lb \Abar_2 + \lb \Ybar_j + \Ybar_k \rb \Abar_1 \rb \Ybar_i^{q-2}
  + \cdots + \lb \Ybar_j + \Ybar_k \rb \Abar_{q-1}.
\end{align}
If follows from \eqref{eq:pixi} that
\begin{align} \label{eq:piQ2}
 \pi \lb \Qbar \rb \in \bC \ld \Xbar_j^{q_j},\Xbar_k^{q_k} \rd.
\end{align}
By projecting \eqref{eq:piQ} by $\pi$,
we obtain
\begin{align}
 \pi \lb \Qbar \rb=\Ybar_i^q-\pi \lb \Pbar \Fbar \rb,
\end{align}
which together with \eqref{eq:PF} and \eqref{eq:piQ2} gives
\begin{align}
 \Abar_l=(-\Ybar_j-\Ybar_k)^l \quad \text{for any } l \in [1,q-1]
\end{align}
and
\begin{align} \label{eq:piQ3}
 \pi(Q)=(-\Ybar_j-\Ybar_k)^q.
\end{align}
It follows from \eqref{eq:piQ2} and \eqref{eq:piQ3}
that
$
 \Ybar_j^a \Ybar_k^{q-a} \in \bC \ld \Xbar_j^{q_j},\Xbar_k^{q_k} \rd
$
for any $a \in [0,q]$.
Since
$
 \Ybar_j^a \Ybar_k^{q-a}
  = \Xbar_j^{a p_j} \Xbar_k^{(q-a) p_k},
$
this implies $q_j \mid p_j$ and $q_k \mid p_k$.
Hence both $p_j/q_j=1/\nu N_j$ and
$p_k/q_k=1/\nu N_k$ are integers.
The product $u := \nu N$ is an integer
by the definitions of $\nu$ and $N$ in \eqref{eq:df}.
The definitions of $N_j$, $N_k$ and $N$
in \eqref{eq:df} and
the fact that both $1/\nu N_j=N/u N_j$ and
$1/\nu N_k=N/u N_k$ are integers
imply $u=1$,
and \pref{pr:nu_N} is proved.
\end{proof}

\begin{corollary} \label{cr:piqi}
We have $q_i=\gcd(p_i,N_i)$ and, in particular, $q_i \mid p_i$ for any $i \in [1,n]$.
\end{corollary}

\begin{proof}
It follows from \eqref{eq:df} and $\nu = 1/N$ that
\begin{align}
 q_i
  =\nu p_i N_i
  =\nu \cdot \lcm(p_i,N_i) \cdot \gcd(p_i,N_i)
  =\nu N \cdot \gcd(p_i, N_i)
  =\gcd(p_i, N_i).
\end{align}
\end{proof}

\begin{corollary} \label{cr:gen}
We have either
\begin{enumerate}[(1)]
 \item \label{it:1}
$R=T$ and $x_i = X_i$ for all $i \in [1,n]$, or
 \item \label{it:2}
$R \ne T$ and
there exists $i \in [1,n]$ such that $q_i=p_i$ and
$
\displaystyle{
 x_j =
\begin{cases}
 X_j^{q_j} & j \ne i, \\
 \prod_{q_k \ne 1} X_k & j = i.
\end{cases}
}
$
\end{enumerate}
\end{corollary}

\begin{proof}
It follows from \eqref{eq:df} and \pref{pr:nu_N} that
\begin{align}
 \frac{1}{N} = 1 - \sum_{i=1}^n \frac{1}{p_i}.
\end{align}
Hence we have
\begin{align}
 (N-1) \vecomega
  &= (N-1) \vecc - (N-1) \sum_{i=1}^n \vecX_i \\
  &= \lc (N-1)-\sum_{i=1}^n \frac{N}{p_i} \rc \vecc + \sum_{i=1}^n \vecX_i \\
  &= \sum_{i=1}^n \vecX_i,
\end{align}
so that
$
 \prod_{i=1}^n X_i \in R_{N-1}.
$
If $R=T$,
then we can set $x_i = X_i$ for all $i \in [1,n]$.
If $R \ne T$,
then we have $|I|=n-1$ by \pref{lm:gen},
and there exists $i \in [1,n]$
such that $x_j = X_j^{q_j}$ for $j \in [1,n] \setminus \{ i \}$.
Note that we have $\Fbar\in \Tbar_\vecc$ and $\vecc=N \vecomega$.
We  can remove $X_k$ such that $q_k=1$ 
from $\prod_{k=1}^n X_k$
to obtain an element
$
 \prod_{q_k \ne 1} X_k,
$
which is one of the generators of $R$.
Hence $x_i=\prod_{q_k \ne 1} X_k$.
Since we have $X_i^{q_i}\in R$ (\pref{lm:power}) and $q_i \mid p_i$ (\pref{cr:piqi}), it follows that $q_i=p_i$.
\end{proof}

\begin{lemma} \label{lm:c}
For any positive integer $n$,
there exist only finitely many sequences
$(p_1, \ldots, p_n)$
of $n$ positive integers
satisfying
\begin{align}
 \sum_{i=1}^n \frac{1}{p_i}+\frac{1}{N}=1,
\end{align}
where
$
 N = \lcm \{ p_i \mid i \in [1,n] \}.
$
\end{lemma}

\begin{proof}
We may assume $p_1\le p_2 \le \cdots \le p_n$.
Then we have $p_1 \le N$ and
\begin{align}
 \frac{n+1}{p_1} \ge \sum_{i=1}^n \frac{1}{p_i}+\frac{1}{N}=1,
\end{align}
so that
\begin{align}
 p_1 \le n+1.
\end{align}
Hence there are only finitely many possibilities for $p_1$.
If we fix $p_1$,
then we have
\begin{align}
 \frac{n}{p_2} \ge \sum_{i=2}^n \frac{1}{p_i}+\frac{1}{N}=1 - \frac{1}{p_1}
\end{align} 
since $p_2 \le N$.
This leaves only finitely many possibilities for $p_2$.
By repeating the same argument,
we can see that there are only finitely many possibilities
for $(p_1, \ldots, p_n)$.
\end{proof}

\begin{proof}[Proof of \pref{th:main}]
For any $n \ge 4$,
\pref{pr:nu_N} and \pref{lm:c}
give a finite list of possible signatures
for hypersurface generalized triangle singularities, and
\pref{th:main} is proved.
\end{proof}

For each signature $\bp=(p_1, \ldots, p_n)$
with $\nu=1/N$,
we can check if $R$ is a hypersurface as follows:
If $q_i=1$ for any $i \in [1,n]$,
then we have $R = T$ and $R$ is a hypersurface.
\begin{comment}
\begin{shaded}
Even if $q_i \ne 1$ for some $i \in [1,n]$,
\pref{cr:piqi} shows that
we need $q_i | p_i$ for any $i \in [1,n]$
in order for $R$ to be a hypersurface.
[$\leftarrow$ $q_i\mid p_i$ follows from $\nu=1/N$]
\end{shaded}
\end{comment}

If $R$ is a hypersurface and $R \ne T$,
then there exists $i \in [1, n]$ such that $q_i=p_i$ and
\begin{align} \label{eq:xj}
 x_j =
\begin{cases}
 X_j^{q_j} & j \ne i, \\
 \prod_{q_k \ne 1} X_k & j = i
\end{cases}
\end{align}
by \pref{cr:gen}.
\begin{comment}
\begin{shaded}
For this to be the case,
we need $q_i=p_i$,
so that $X_i^{q_i}$ can be written in terms of
$X_j^{q_j}$ for $j \in [1,n] \setminus \{ i \}$.
\end{shaded}
\end{comment}
Assume that there exists $i \in [1,n]$ such that $q_i=p_i$.
Fix any such $i$ and define $\{ x_j \}_{j=1}^n$ by \eqref{eq:xj}.
Let $R'$ be the subring of $T$
generated by $\{ x_j \}_{j=1}^n$,
so that $R$ is a hypersurface
if and only if $R = R'$.
It follows from $\nu=1/N$ that $q_j \mid p_j$ for any $j \in [1,n]$ (cf.\ \pref{cr:piqi}).
Hence $Y_j := X_j^{p_j}$ is contained in $R'$ for any $j \in [1,n]$.
Note that any element of $T_{\vecv}$ for $\vecv \in L$
can be written as the product $M(\vecv) P$,
where $M(\vecv):=\varphi(\Mbar(\vecv))$ is the image of $\Mbar(\vecv)$ defined by \eqref{eq:Mbar},
and $P$ is a homogeneous element of $\bC[Y_1, \ldots, Y_n]/(Y_1+\cdots+Y_n)$.
Since $M((k+N) \vecomega) = M(k \vecomega)$ for any $k \in \bZ$,
the ring $R$ is generated by $\{ Y_j \}_{j=1}^{n}$ and $\{ M(k\vecomega) \}_{k=0}^{N-1}$.
Therefore we have $R = R'$
if and only if $M(k \vecomega) \in R'$ for $0 \leq k \leq N-1$.
\pref{tb:dim3} is obtained in this way.

\section{Proof of \pref{th:isolated}}
 \label{sc:isolated}

We keep the same notations as in \pref{sc:main}.
Given a signature $\bp = (p_1, \ldots, p_n)$,
we define a group $G \subset \GL_n(\bC)$ by
\begin{equation}
 G = \lc \diag(\alpha_1, \ldots, \alpha_n) \relmid
  \alpha_1^{p_1}=\cdots=\alpha_n^{p_n} = \prod_{i=1}^{n} \alpha_i=1 \rc.
\end{equation}
The group $G$ acts naturally on $T$ in such a way that
$\diag(\alpha_1, \ldots, \alpha_n) \in G$ maps
$X_i \in T$ to $\alpha_i X_i$ for $i \in [1,n]$.

\begin{lemma} \label{lm:invariant_ring}
If $n\geq 4$ and $R$ is a hypersurface,
then $R$ coincides with the invariant ring $T^G$.
\end{lemma}

\begin{proof}
We have
$
 N \vecomega = N \nu \vecc,
$
which is equal to $\vecc$ by \pref{pr:nu_N}.
This shows that
$
 \bZ \vecomega \supset \bZ \vecc.
$
Note that
\begin{align}
 L / \bZ \vecc \cong \bigoplus_{i=1}^n \bZ \vecX_i / ( p_i \vecX_i )
\end{align}
and
\begin{align}
 \vecomega \equiv - \sum_{i=1}^n \vecX_i \mod \vecc.
\end{align}
It follows that
\begin{align}
 L / \bZ \vecomega
  &\cong \left. \lb \bigoplus_{i=1}^n \bZ \vecX_i \rb \right/ \lb p_i \vecX_i, \sum_{i=1}^n \vecX_i \rb.
\end{align}
This allows us to identify $L/\bZ \vecomega$
with the group of characters of $G$,
so that the ring $R$, which is the Veronese subring over $\vecomega$, is exactly
the $G$-invariant part of $T$.
\end{proof}

\begin{proposition} \label{pr:isolated}
If $n\geq 4$ and $R$ is a hypersurface,
then $R$ has an isolated singularity if and only if $R=T$.
\end{proposition}

\begin{proof}
The `if' part is clear since $T$ has an isolated singularity.
To prove the `only if' part,
assume that $R \subsetneq T$.
Then we have $s := \gcd(p_i,p_j) \ne 1$
for some $1 \le i < j \le n$ by \pref{lm:power} and \pref{cr:piqi}.
Define a subset of
\begin{align}
 \Spec T = \{ (X_1, \ldots, X_n) \in \bA^n \mid X_1^{p_1} + \cdots + X_n^{p_n} = 0 \}
\end{align}
by
\begin{align}
 P = \{ (X_1, \ldots, X_n) \in \Spec T \mid X_i = X_j = 0 \text{ and }
  X_k \ne 0 \text{ for any } k \ne i, j \}.
\end{align}
Then the stabilizer subgroup of any point in $P$
with respect to the action of $G$
is given by
\begin{align}
 \lc \diag(\alpha_1,\ldots,\alpha_n) \relmid
 \alpha_i^{s}=1, \ 
 \alpha_i \alpha_j=1 \text{ and }
 \alpha_k=1 \text { for any } k \ne i, j \rc.
\end{align}
This is isomorphic to a cyclic subgroup of $\SL_2(\bC)$,
so that $\Spec R = (\Spec T)/G$ has a non-isolated family
of $A_{s-1}$-singularities along $P/G$.
\end{proof}

Now we prove \pref{th:isolated}:

\begin{proof}[Proof of \pref{th:isolated}]
To prove the `if' part,
assume that we have \eqref{eq:fraction}.
Then we have
\begin{align}
 \nu
  := 1 - \sum_{i=1}^n \frac{1}{p_i}
  = \frac{1}{\prod_{i=1}^n p_i}.
\end{align}
It follows that for any $i \in [1,n]$,
we have
\begin{align}
 q_i
  &:= \nu p_i N_i \\
  &= \frac{1}{\prod_{i=1}^n p_i} \cdot p_i \cdot \lcm \lc p_j \relmid j \in [1,n] \setminus \{ i \} \rc \\
  &\le \frac{1}{\prod_{i=1}^n p_i} \cdot p_i \cdot \prod_{j \in [1,n] \setminus \{ i \}} p_j \\
  &= 1.
\end{align}
This implies $q_i=1$
since $q_i$ is a positive integer
by definition.
Hence we have $R=T$ (\pref{lm:power}),
which clearly has an isolated singularity at the origin.

To prove the `only if' part,
assume that $R$ has an isolated hypersurface singularity.
Then \pref{pr:isolated} shows $R=T$,
which implies $q_i=1$ for any $i \in [1,n]$
by \pref{lm:power}.
Then one has $\nu  = 1/N$ and $\gcd(p_i, N_i) = 1$ for any $i \in [1,n]$
by \eqref{eq:nuN},
which implies $N = \lcm \lc p_i \relmid i \in [1,n] \rc = \prod_{i=1}^n p_i$
and \eqref{eq:fraction}
by \eqref{eq:df}.
\end{proof}

\section{Proof of \pref{th:a}}
 \label{sc:a}

Since we assume that $n\geq 4$ and $R$ is a hypersurface, we have
\begin{equation} \label{eq:nu_N_is_1}
 \nu:=1-\sum_{i=1}^{n}\frac{1}{p_i}=\frac{1}{N}
\end{equation}
 by \pref{pr:nu_N}.
Define a function
$
 m \colon \bZ \to \bZ
$
by
\begin{align} \label{eq:m}
 m(k)
 := \ell(k \vecomega)
 =k-\sum_{i=1}^{n} \myceil{\frac{k}{p_i}},
\end{align}
where the function
$
 \ell \colon L \to \bZ
$
is defined by \eqref{eq:ka}.
Recall that we have
\begin{equation}
 R_k=\Mbar(k \vecomega) \cdot
  \lb \bC[\Ybar_1,\ldots,\Ybar_n] / \lb \Ybar_1+\cdots+\Ybar_n \rb \rb_{m(k)},
\end{equation}
 where $\Mbar(k \vecomega)$ is defined by \eqref{eq:Mbar}.
Therefore the Hilbert series of $R$ is given by
\begin{align} \label{g_j}
 F(t)
  := \sum_{k=0}^\infty (\dim_\bC R_k) \, t^k
  = \sum_{k=0}^\infty c(k) t^n,
\end{align}
where
\begin{align}
 c(k) = \begin{cases}
  \displaystyle{\binom{m(k)+n-2}{n-2}} & m(k) \geq 0, \\
  0 & m(k)<0.
 \end{cases}
\end{align}
We can write
\begin{align}
 F(t) =\sum_{j=0}^{N-1} g_j(t) t^j,
\end{align}
where
\begin{align}
 g_j(t) = \sum_{k=0}^\infty c(j+Nk) t^{Nk}.
\end{align}
It follows from \eqref{eq:nu_N_is_1} and \eqref{eq:m} that
\begin{align}
 m(j+N k)
 =m(j)+\nu N k=m(j)+k.
\end{align}
Hence we have
\begin{align}
 c(j+Nk) = \begin{cases}
  \displaystyle{\binom{m(j)+k+n-2}{n-2}} & m(j)+k \geq 0, \\
  0 & m(j)+k<0.
 \end{cases}
\end{align}
Therefore, by the following \pref{lm:value_m_k},
\begin{align} \label{eq:gj}
 g_j(t) = \sum_{k=0}^\infty \binom{k+n-2}{n-2} t^{N(k-m(j))}.
\end{align}

\begin{lemma} \label{lm:value_m_k}
We have $m(0)=0$, $m(1)=-(n-1)$, and
\begin{align}
 -(n-2) \leq m(k) \leq 0
\end{align}
for $2 \leq k \leq N-1$.
\end{lemma}

\begin{proof}
It is clear that $m(0)=0$ and $m(1)=-(n-1)$.
For $k\leq N-1$, we have
\begin{align*}
 m(k)
  \le k - \sum_{i=1}^n \frac{k}{p_i}
  = \nu k
  = \frac{k}{N}
  < 1,
\end{align*}
so that $m(k)\leq 0$.
On the other hand, for $k\geq 2$, we have
\begin{align*}
 m(k)
  \ge k - \sum_{i=1}^n \lb \frac{k}{p_i} + \frac{p_i - 1}{p_i} \rb
  = \nu k-n-(\nu - 1)
  = \nu(k-1)-(n-1)
  > -(n-1),
\end{align*}
so that $m(k) \ge -(n-2)$.
\end{proof}


\begin{lemma} \label{lm:deg_F}
$\lb 1-t^N \rb^{n-1} F(t)$ is a polynomial of degree $(n-1)N+1$.
\end{lemma}

\begin{proof}
By \pref{eq:gj}, we have
\begin{align}
 \lb 1-t^N \rb^{n-1} F(t)
 &= \lb 1-t^N \rb^{n-1} \sum_{j=0}^{N-1} g_j(t) t^j \\
 &= \lb 1-t^N \rb^{n-1} \sum_{j=0}^{N-1} t^j
  \sum_{k=0}^\infty \binom{k+n-2}{n-2} t^{N(k-m(j))} \\
 &= \sum_{j=0}^{N-1} t^{j-N \cdot m(j)} \lb 1-t^N \rb^{n-1}
  \sum_{k=0}^\infty \binom{k+n-2}{n-2} t^{N k} \\
 &= \sum_{j=0}^{N-1} t^{j-N \cdot m(j)},
  \label{F_sum}
\end{align}
which is a polynomial of degree $(n-1)N+1$ by \pref{lm:value_m_k}.
\end{proof}

Now we prove \pref{th:a}:

\begin{proof}[Proof of \pref{th:a}]
If $R$ is generated by elements
of degrees $a_1, \ldots, a_n$
with one relation of degree $h$,
then the Hilbert series of $R$ is given by
\begin{align} \label{F_prod}
 F(t)=\frac{1-t^h}{\prod_{i=1}^n \lb 1-t^{a_i} \rb}.
\end{align}
By \pref{lm:deg_F}, we have
\begin{align}
 h=\sum_{i=1}^n a_i + 1.
\end{align}
This concludes the proof of \pref{th:a}. 
\end{proof}

\begin{remark}
By (\ref{F_prod}), we have
\begin{align}
 \lim_{t \to 1} (1-t)^{n-1} F(t)
 =\lim_{t \to 1} \frac{1+t+\cdots+t^{h-1}}{\prod_{i=1}^n (1+t+\cdots+t^{a_i-1})}
 =\frac{h}{\prod_{i=1}^n a_i}.
\end{align}
Similarly, by (\ref{F_sum}), we have
\begin{align}
 \lim_{t \to 1} (1-t)^{n-1} F(t)
  &= \lim_{t \to 1} \frac{(1-t)^{n-1}}{(1-t^N)^{n-1}} \sum_{j=0}^{N-1} t^{j-N \cdot m(j)} \\
  &=\frac{1}{N^{n-1}} \cdot N \\
  &=\frac{1}{N^{n-2}}.
\end{align}
Hence the following equality holds:
\begin{align}
 \frac{h}{\prod_{i=1}^n a_i}=\frac{1}{N^{n-2}}.
\end{align}
A related discussion can be found
in \cite[Theorem 2.6]{MR586722}.
\end{remark}

\bibliographystyle{amsalpha}
\bibliography{bibs}

\noindent
Kenji Hashimoto

School of Mathematics,
Korea Institute for Advanced Study,
85 Hoegiro,
Dongdaemun-gu,
Seoul,
130-722,
Korea

{\em e-mail address}\ : \ hashimoto@kias.re.kr

\ \vspace{-7mm} \\

\noindent
Hwayoung Lee

Department of Mathematical Sciences,
Research Institute of Mathematics,
Seoul National University,
Gwanak-ro 1,
Gwanak-gu,
Seoul
151-747,
Korea

{\em e-mail address}\ : \ hlee014@snu.ac.kr 

\ \vspace{-7mm} \\

\noindent
Kazushi Ueda

Department of Mathematics,
Graduate School of Science,
Osaka University,
Machikaneyama 1-1,
Toyonaka,
Osaka,
560-0043,
Japan.

{\em e-mail address}\ : \  kazushi@math.sci.osaka-u.ac.jp
\ \vspace{0mm} \\

\end{document}